\documentclass[12pt, reqno]{amsart}
\usepackage{amsmath, amsthm, amscd, amsfonts, amssymb, graphicx, color}
\usepackage[bookmarksnumbered, colorlinks, plainpages]{hyperref}
\input{mathrsfs.sty}

\newtheorem{theorem}{Theorem}[section]

\newtheorem{corollary}[theorem]{Corollary}
\theoremstyle{definition}

\theoremstyle{remark}
\newtheorem{remark}[theorem]{Remark}
\numberwithin{equation}{section}

\newcommand{\m}{ \frac{1}{b-a}\int_a^b f(x)dx }

\newcommand{\mf}{m_f(a,b)}
\newcommand{\h}{\mathscr{H}}
\newcommand{\+}{\mathbb{B}(\h)_+}

\begin{document}

\title{Further refinements of the Heinz inequality}
\author[R. Kaur, M.S. Moslehian, M. Singh, C. Conde]{ Rupinderjit Kaur$^1$, Mohammad Sal Moslehian$^2$, Mandeep Singh$^1$ and Cristian Conde$^3$}

\address{$^1$ Department of Mathematics, Sant Longowal Institute of Engineering and Technology, Longowal-148106, Punjab, India}
\email{rupinder\_grewal\_86@yahoo.co.in} \email{msrawla@yahoo.com}
\address{$^2$ Department of Pure Mathematics, Center of Excellence in
Analysis on Algebraic Structures (CEAAS), Ferdowsi University of
Mashhad, P.O. Box 1159, Mashhad 91775, Iran}
\address{$^3$ Instituto de Ciencias, Universidad Nacional de Gral. Sarmiento, J.
M. Gutierrez 1150, (B1613GSX) Los Polvorines and Instituto Argentino
de Matemática ``Alberto P. Calder\'on", Saavedra 15 3º piso,
(C1083ACA) Buenos Aires, Argentina} \email{cconde@ungs.edu.ar}

\email{moslehian@ferdowsi.um.ac.ir and moslehian@member.ams.org}



\begin{abstract} The celebrated Heinz inequality asserts that $
2|||A^{1/2}XB^{1/2}|||\leq
|||A^{\nu}XB^{1-\nu}+A^{1-\nu}XB^{\nu}|||\leq |||AX+XB|||$ for $X
\in \mathbb{B}(\mathscr{H})$, $A,B\in \+$, every unitarily invariant
norm $|||\cdot|||$ and $\nu \in [0,1]$. In this paper, we present
several improvement of the Heinz inequality by using the convexity
of the function $F(\nu)=|||A^{\nu}XB^{1-\nu}+A^{1-\nu}XB^{\nu}|||$,
some integration techniques and various refinements of the
Hermite--Hadamard inequality. In the setting of matrices we prove
that
\begin{eqnarray*}
&&\hspace{-0.5cm}\left|\left|\left|A^{\frac{\alpha+\beta}{2}}XB^{1-\frac{\alpha+\beta}{2}}+A^{1-\frac{\alpha+\beta}{2}}XB^{\frac{\alpha+\beta}{2}}\right|\right|\right|\leq\frac{1}{|\beta-\alpha|}
\left|\left|\left|\int_{\alpha}^{\beta}\left(A^{\nu}XB^{1-\nu}+A^{1-\nu}XB^{\nu}\right)d\nu\right|\right|\right|\nonumber\\
&&\qquad\qquad\leq
\frac{1}{2}\left|\left|\left|A^{\alpha}XB^{1-\alpha}+A^{1-\alpha}XB^{\alpha}+A^{\beta}XB^{1-\beta}+A^{1-\beta}XB^{\beta}\right|\right|\right|\,,
\end{eqnarray*}
for real numbers $\alpha, \beta$.
\end{abstract}
\maketitle

\section{Introduction}

Let $\mathbb{B}(\mathscr{H})$ denote the $C^*$-algebra of all
bounded linear operators acting on a complex separable Hilbert space
$(\mathscr{H},\langle \cdot,\cdot\rangle)$. In the case when $\dim
\mathscr{H} = n$, we identify $\mathbb{B}(\mathscr{H})$ with the
full matrix algebra $\mathcal{M}_n$ of all $n\times n$ matrices with
entries in the complex field. The cone of positive operators is
denoted by ${\mathbb B}({\mathscr H})_+$. A unitarily invariant norm
$\left|\left|\left|\cdot\right|\right|\right|$ is defined on a norm
ideal $\mathfrak{J}_{\left|\left|\left|\cdot\right|\right|\right|}$
of $\mathbb{B}(\mathscr{H})$ associated with it and has the property
$\left|\left|\left| UXV\right|\right|\right|=\left|\left|\left|
X\right|\right|\right|$, where $U$ and $V$ are unitaries and $X \in
\mathfrak{J}_{\left|\left|\left| .\right|\right|\right|}$. Whenever
we write $|||X|||$, we mean that $X \in
\mathfrak{J}_{\left|\left|\left| \cdot\right|\right|\right|}$. The
operator norm on $\mathbb{B}(\mathscr{H})$ is denoted by
$\|\cdot\|$.

The arithmetic--geometric mean inequality for two positive real numbers $a,b$ is $\sqrt{ab}\leq (a+b)/2$, which has been generalized in the context of bounded linear operators as follows. For  $A,B\in \+$ and an unitarily invariant norm $|||\cdot|||$ it holds that
\begin{eqnarray*}
2|||A^{1/2}XB^{1/2}|||\leq |||AX+XB|||.
\end{eqnarray*}

For $0\leq \nu \leq 1$ and two nonnegative real numbers $a$ and $b$, the {\it Heinz mean} is defined as
$$
H_{\nu}(a,b)=\frac{a^{\nu}b^{1-\nu}+a^{1-\nu}b^{\nu}}{2}.
$$
The function $H_{\nu}$ is symmetric about the point
$\nu =\frac 12$. Note that $H_0(a,b)=H_1(a,b)=\frac{a+b}{2}$, $H_{1/2}(a,b)=\sqrt{ab}$ and
\begin{eqnarray}\label{Heinz}
H_{1/2}(a,b)\leq H_{\nu}(a,b)\leq H_0(a,b)
\end{eqnarray}
for $0\leq \nu \leq 1$, i.e., the Heinz means interpolates between the geometric mean and the arithmetic mean. The generalization of \eqref{Heinz}
in $B(\h)$ asserts that for operators $A,B, X$ such that $A,B\in \+$, every unitarily invariant norm $|||\cdot|||$ and $\nu \in [0,1]$ the following double inequality due to Bhatia and Davis \cite{1} holds
\begin{eqnarray}\label{Heinz2}
2|||A^{1/2}XB^{1/2}|||\leq |||A^{\nu}XB^{1-\nu}+A^{1-\nu}XB^{\nu}|||\leq |||AX+XB|||.
\end{eqnarray}
Indeed, it has been proved that $F(\nu)=|||A^{\nu}XB^{1-\nu}+A^{1-\nu}XB^{\nu}|||$ is a convex function of $\nu$ on $[0,1]$
with symmetry about $\nu=1/2$, which attains its minimum there at and its maximum at $\nu =0$ and $\nu=1$.

The second part of the previous inequality is one of the most
essential inequalities in the operator theory, which is called {\it
the Heinz inequality};  see \cite{He}. The proof given by Heinz
\cite{4} is based on the complex analysis and is somewhat
complicated. In \cite{Mc}, McIntosh showed that the Heinz inequality
is a consequence of the following inequality
\begin{eqnarray*}
\left\Vert A^*AX+XBB^{*
}\right\Vert \geq 2\left\Vert AXB\right\Vert\,,
\end{eqnarray*}
where $A,B,X\in \mathbb{B}(\mathscr{H})$. In the literature, the
above inequality is called the {\it arithmetic--geometric mean
inequality}. J.I. Fujii, M. Fujii, T. Furuta and M. Nakamoto
\cite{fuj1} proved that the Heinz inequality is equivalent to
several other norm inequalities such as the {\it
Corach--Porta--Recht inequality} $\|AXA^{-1}+A^{-1}XA\|\geq2\|X\|$,
where $A$ is a selfadjoint invertible operator and $X$ is a
selfadjoint operator; see also \cite{cms}. Audenaert \cite{0} gave a
singular value inequality for Heinz means by showing that if $A, B
\in \mathcal{M}_n$ are positive semidefinite and $0 \leq \nu \leq
1$, then $s_j(A^{\nu}B^{1-\nu}+A^{1-\nu}B^{\nu})\leq s_j(A+B)$ for
$j=1, \cdots ,n$, where $s_j$ denotes the $j$th singular value.
Also, Yamazaki \cite{11} used the classical Heinz inequality $
\|AXB\|^r \|X\|^{1-r}\ge\|A^r X B^r\|\,\,(A, B, X \in
\mathbb{B}(\mathscr{H}), A \geq 0, B \geq 0, r\in [0,1])$ to
characterize the chaotic order relation and to study isometric
Aluthge transformations.

For a detailed study of these and associated norm inequalities along
with their history of origin, refinements and applications, one may
refer to \cite{1, 2, 3, 5, 6, 7, 8}.

It should be noticed that $F(1/2)\leq F(\nu) \leq
\frac{F(0)+F(1)}{2}$ provides a refinement to the Jensen inequality
$F(1/2) \leq \frac{F(0)+F(1)}{2}$ for the function $F$. Therefore it
seems quite reasonable to obtain a new refinement of \eqref{Heinz2}
by utilizing a refinement of Jensen's inequality. This idea was
recently applied by Kittaneh \cite{9} in virtue of the
Hermite--Hadamard inequality \eqref{Hadamard}.

One of the purposes of the present article is to obtain some new
refinements of \eqref{Heinz2}, from different refinements of
inequality \eqref{Hadamard}. We also aim to give a unified study and
further refinements to the recent works for matrices.

\section{The Hermite--Hadamard inequality and its refinements}

For a convex function $f$, the double inequality
\begin{eqnarray}\label{Hadamard}
f\left(\frac{a+b}{2}\right)\leq \frac{1}{b-a}\int_a^b f(x)dx\leq \frac{f(a)+f(b)}{2}
\end{eqnarray}
is known as the {\it Hermite--Hadamard} (H-H) inequality. This
inequality was first published by Hermite in 1883 in an elementary
journal and independently proved in 1893 by Hadamard. It gives us an
estimation of the mean value of the convex function $f$; see
\cite{K} and \cite{MOS}.

There is an extensive amount of literature devoted to this simple
and nice result, which has many applications in the theory of
special means from which we would like to refer the reader to
\cite{[PPT]}. Interestingly, each of two sides of the H-H inequality
characterizes convex functions. More precisely, if $J$ is an
interval and $f: J \to  \mathbb{R}$ is a continuous function, whose
restriction to every compact subinterval $[a, b]$ verifies the first
inequality of \eqref{Hadamard} then $f$ is convex. The same works
when the first inequality is replaced by the second one.

Applying the H-H inequality, one can obtain the well-known geometric--logarithmic--arithmetic inequality
$$
H_{1/2}(a,b)\leq L(a,b)\leq H_0(a,b),
$$
where $L(a,b)=\int_0^1 a^t b^{1-t}dt.$ An operator version of this
has been proved by Hiai and  Kosaki \cite{6}, which says that for
$A, B\in\+$,
$$
|||A^{1/2}XB^{1/2}|||\leq \left|\left|\left|\int_0^1
A^{\nu}XB^{1-\nu}d\nu\right|\right|\right|\leq \frac
12|||AX+XB|||\,,
$$
which is another refinement of the arithmetic--geometric operator inequality.

Throughout this paper we will use the following notation: For $a,b\in \mathbb{R}$ and $t\in[0,1]$, let
$$m_f(a,b)=\m,$$
and
$$[a,b]_t=(1-t)a+tb.$$

If $f$ is an integrable function on $[a,b]$ then $$\m=\int_0^1f(ta+(1-t)b)dt=\int_0^1f(tb+(1-t)a)dt,$$
and if $f$ is convex on $[a,b]$ we get
$$
\m=\int_0^1 F_{(a,b)}(t)dt,
$$
where $F_{(a,b)}(t)=\frac 12 \left(f\left(a +
\frac{t(b-a)}{2}\right)+ f\left(b - \frac{t(b-a)}{2}\right)\right)$;
see \cite[Theorem 1.2]{A}.

In this section we collect various refinements of the H-H inequality
for convex functions.

\begin{theorem}\label{HG} \cite{[Dr92], [YH]}
If $f:[a,b]\to \mathbb{R}$ is a convex function and $H_t, G_t$ are defined on $[0,1]$ by
$$
H_t(a,b)=\frac{1}{b-a}\int_a^b f\left(\left[\frac{a+b}{2},x\right]_t\right)dx,
$$
and
$$
G_t(a,b)=\frac{1}{2(b-a)}\int_a^b [f(\left[x,a\right]_t)+f(\left[x,b\right]_t)]dx,
$$
then $H_t$ and $G_t$ are convex, increasing and
\begin{eqnarray}\label{DragomirH}
f\left(\frac{a+b}{2}\right)=H_0(a,b)\leq H_t(a,b)\leq H_1(a,b)=\mf,
\end{eqnarray}
\begin{eqnarray}\label{DragomirG}
\mf=G_0(a,b)\leq G_t(a,b)\leq G_1(a,b)=\frac{f(a)+f(b)}{2}
\end{eqnarray}
for all $t\in [0,1]$. Furthermore,
\begin{eqnarray*}
f\left(\frac{a+b}{2}\right)&\leq&\frac{2}{b-a}\int_{\frac{(3a+b)}{4}}^{\frac{(a+3b)}{4}} f(x)dx
\leq \int_0^1 H_t(a,b)dt \nonumber \\ &\leq& \frac 12 \left( f\left(\frac{a+b}{2}\right)+\mf\right)\leq \mf
\end{eqnarray*}
and
\begin{eqnarray}\label{DragomirK}
\frac{2}{b-a}\int_{\frac{(3a+b)}{4}}^{\frac{(a+3b)}{4}} f(x)dx
&\leq&  \frac 12\left( f\left(\frac{3a+b}{4}\right)+f
\left(\frac{a+3b}{4}\right)\right)\leq \int_0^1 G_t(a,b)dt \nonumber \\
&\leq& \frac 12 \left( f\left(\frac{a+b}{2}\right)+\frac{f(a)+f(b)}{2}\right)\nonumber \\
&\leq& \frac{f(a)+f(b)}{2}.
\end{eqnarray}
\end{theorem}

\begin{remark}
\begin{enumerate}
\item  From \eqref{DragomirK} we get that
\begin{eqnarray*}
\mf\leq
\frac12\left(f\left(\frac{a+b}{2}\right)+\frac{f(a)+f(b)}{2}\right)\leq
\frac{f(a)+f(b)}{2},
\end{eqnarray*}
which is the well-known Bullen's inequality; see \cite[p. 140]{[PPT]}. As an immediate consequence, from the previous inequality, we  note that the first inequality is stronger than the
second one in \eqref{Hadamard}, i.e.
$$
\mf - f\left(\frac{a+b}{2}\right)\leq \frac{f(a)+f(b)}{2}-\mf.
$$
\item We note some properties of $H_t$ and $G_t$ useful in the next sections. For $\mu \in [0,1]$ we get
\begin{enumerate}
\item $H_t(\mu,1-\mu)=\frac{1}{1-2\mu}\int_\mu^{1-\mu} f\left(\left[\frac{1}{2},x\right]_t\right)dx=\frac{1}{2\mu-1}\int_{1-\mu}^{\mu} f\left(\left[\frac{1}{2},x\right]_t\right)dx=H_t(1-\mu, \mu).$
\item $G_t(\mu,1-\mu)=\frac{1}{2(1-2\mu)}\int_{\mu}^{1-\mu} [f(\left[x,\mu\right]_t)+f(\left[x,1-\mu\right]_t)]dx=G_t(1-\mu,\mu).$
\end{enumerate}
\end{enumerate}
\end{remark}

Recently, the following result was proved:

\begin{theorem}\label{W}\cite{[W]} If $f$ is a convex function defined on an interval $J$, $a,b \in J^{\circ}$ with $a<b$ and the mapping $T_t$ is defined by
$$
T_t(a,b)=\frac 12 \left(f\left(\frac{1+t}{2} a+
\frac{1-t}{2}b\right)+f\left(\frac{1-t}{2} a
+\frac{1+t}{2}b\right)\right),
$$
then $T_t$ is convex and increasing on $[0,1]$ and
\begin{eqnarray*}
f\left(\frac{a+b}{2}\right)\leq T_{\eta}(a,b) \leq
T_{\xi}(a,b)\leq T_{\lambda}(a,b)\leq \frac{f(a)+f(b)}{2},
\end{eqnarray*}
for all $\eta \in (0, \xi),
\lambda \in (\xi, 1)$, where $T_{\xi}(a,b)=\mf.$
\end{theorem}
In \cite{EF}, the author asked whether for a convex function $f$ on an interval $J$ there exist real numbers $l$, $L$ such that
$$
f\left(\frac{a+b}{2}\right)\leq l\leq \frac{1}{b-a}\int_a^b f(x)dx\leq L \leq \frac{f(a)+f(b)}{2}\,.
$$
An affirmative answer to this question is given as follows.
\begin{theorem} \cite{EF}\label{Far}
 Assume that $f:[a,b]\to\mathbb{R}$ is a convex function. Then
\begin{eqnarray}\label{Farissi}
 f\left(\frac{a+b}{2}\right)\leq l(\lambda)\leq \frac{1}{b-a}\int_a^b f(x)dx\leq L(\lambda) \leq \frac{f(a)+f(b)}{2}
\end{eqnarray}
for all $\lambda \in [0,1]$, where
$$
l(\lambda)=\lambda f\left(\frac{\lambda b +(2-\lambda)a}{2}\right)+(1-\lambda) f\left(\frac{(1+\lambda) b +(1-\lambda)a}{2}\right)
$$
 and
$$
L(\lambda)=\frac12 (f(\lambda b +(1-\lambda)a)+\lambda f(a)+(1-\lambda)f(b)).
$$

\end{theorem}

\begin{remark}

Applying inequality \eqref{Farissi} for $\lambda=\frac12$ we get
\begin{eqnarray*}
 f\left(\frac{a+b}{2}\right)&\leq& \frac 12\left( f\left(\frac{3a+b}{4}\right)+f
\left(\frac{a+3b}{4}\right)\right)\leq \mf\nonumber \\&\leq&
\frac12\left(f\left(\frac{a+b}{2}\right)+\frac{f(a)+f(b)}{2}\right)
\leq \frac{f(a)+f(b)}{2}.
\end{eqnarray*}
This result has been obtained by Akkouchi in \cite{A}.
\end{remark}

\section{Refinements of the Heinz inequality for operators}
In this section we use the convexity of
$F(\nu)=|||A^{\nu}XB^{1-\nu}+A^{1-\nu}XB^{\nu}|||;\,\, \nu \in
[0,1]$ and the different refinements of inequality \eqref{Hadamard}
described in the previous section.
\begin{theorem}
Let $A,B, X$ be operators such that $A, B\in \+$. Then for any $t,\mu \in [0,1]$ and any unitary invariant norm $|||\cdot|||$,
\begin{eqnarray}
2|||A^{1/2}XB^{1/2}|||&\leq& \frac{1}{1-2\mu}\int_{\mu}^{1-\mu}
F([ 1/2, x]_t) dx\nonumber \\
&\leq& \frac{1}{1-2\mu}\int_{\mu}^{1-\mu} |||A^{x}XB^{1-x}+A^{1-x}XB^{x}||| dx \nonumber\\
&\leq& \frac{1}{2(1-2\mu)}\int_{\mu}^{1-\mu} [F(\left[x,\mu\right]_t)+F(\left[x,1-\mu\right]_t)]dx  \nonumber \\
&\leq& |||A^{\mu}XB^{1-\mu}+A^{1-\mu}XB^{\mu}||| \ \nonumber
\end{eqnarray}
\end{theorem}
\begin{proof}
For $\mu\neq \frac 12$ the inequalities follows by applying inequalities
\eqref{DragomirH} and \eqref{DragomirG} on the interval $[\mu, 1-\mu]$ if $0\leq
\mu < \frac{1}{2}$ or $[1-\mu,\mu]$ if $\frac 12 <\mu \leq 1$.
\noindent Finally
$$\lim\limits_{\mu \to \frac 12}\frac{1}{2(1-2\mu)}\int_{\mu}^{1-\mu} \left(F(\left[x,\mu\right]_t)+F(\left[x,1-\mu\right]_t)\right)dx=2|||A^{1/2}XB^{1/2}|||$$ completes the proof.
\end{proof}

Applying Theorem \ref{HG} to the function $F$ on the interval $[\mu, \frac 12]$ or $[\frac 12, \mu]$ for $\mu\in [0,1]$ we obtain the following refinement of \cite[Theorem 2 and Corollary 1]{9}.

\begin{theorem}
Let $A,B, X$ be operators such that $A, B\in \+.$ Then for every $\mu\in [0,1]$ and every unitarily invariant norm $|||\cdot|||$,
\begin{eqnarray*}
&2&\hspace{-0.3cm}|||A^{1/2}XB^{1/2}|||\leq |||A^{\frac{2\mu+1}{4}}XB^{\frac{3-2\mu}{4}}+A^{\frac{3-2\mu}{4}}XB^{\frac{2\mu+1}{4}}|||  \\
&\leq&\frac{4}{1-2\mu}\int_{\frac{(6\mu+1)}{8}}^{\frac{(2\mu+3)}{8}} |||A^{x}XB^{1-x}+A^{1-x}XB^{x}|||dx
\leq \int_0^1 H_t(1/2,\mu)dt \nonumber \\
&\leq& \frac 12 |||A^{\frac{2\mu+1}{4}}XB^{\frac{3-2\mu}{4}}+A^{\frac{3-2\mu}{4}}XB^{\frac{2\mu+1}{4}}|||+\frac{1}{1-2\mu}\int_{\mu}^{1/2}F(x)dx \nonumber \\
&\leq& \frac{2}{1-2\mu}\hspace{-0.1cm}\int_{\mu}^{1/2}|||A^{x}XB^{1-x}+A^{1-x}XB^{x}||| dx=G_0(1/2,\mu)
\leq \int_0^1 G_t(1/2,\mu)dt \nonumber \\
&\leq& \frac 12\left(|||A^{\frac{2\mu+1}{4}}XB^{\frac{3-2\mu}{4}}+A^{\frac{3-2\mu}{4}}XB^{\frac{2\mu+1}{4}}|||+|||A^{\mu}XB^{1-\mu}+A^{1-\mu}XB^{\mu}|||+F(1/2)\right) \nonumber\\
&\leq& \frac 12 |||A^{\mu}XB^{1-\mu}+A^{1-\mu}XB^{\mu}|||+|||A^{1/2}XB^{1/2}|||  \\
&\leq& |||A^{\mu}XB^{1-\mu}+A^{1-\mu}XB^{\mu}|||\,.
\end{eqnarray*}

\end{theorem}
Now, we have the following refinement of the first part of the the
Heinz inequality via certain sequences.

\begin{theorem}
Let $A,B, X$ be operators such that $A, B\in \+$ and for $n\in \mathbb{N}_{0}$ ,
\begin{eqnarray*}
x_n(F,a,b)=\frac{1}{2^n}\sum_{i=1}^{2^n} F\left(a+\left(i-\frac
12\right)\frac{b-a}{2^n}\right),
\end{eqnarray*}
\begin{eqnarray*}
y_n(F,a,b)=\frac{1}{2^n}\left(\frac{F(a)+F(b)}{2}+\sum_{i=1}^{2^n-1}
F\left([a,b]_{\frac{i}{2^n}}\right)\right).
\end{eqnarray*}

Then
 \begin{enumerate}
 \item For $\mu \in [0,1/2]$ and for every unitarily invariant norm $|||\cdot|||$,
 \begin{eqnarray*}
 2|||A^{1/2}XB^{1/2}|||&=&x_0(F,\mu,1-\mu)\leq \cdots\leq x_n(F,\mu,1-\mu) \\
 &\leq& \frac{1}{1-2\mu}\int_{\mu}^{1-\mu} |||A^{x}XB^{1-x}+A^{1-x}XB^{x}||| dx  \\
&\leq& y_n(F,\mu,1-\mu)\leq \cdots\leq y_0(F,\mu,1-\mu)=F(\mu)
 \end{eqnarray*}
 \item For $\mu \in [1/2,1]$ and for every unitarily invariant norm $|||\cdot|||$,
 \begin{eqnarray*}
 2|||A^{1/2}XB^{1/2}|||&=&x_0(F,1-\mu,\mu)\leq \cdots\leq x_n(F,1-\mu,\mu) \\
 &\leq& \frac{1}{2\mu-1}\int_{1-\mu}^{\mu} |||A^{x}XB^{1-x}+A^{1-x}XB^{x}||| dx  \\
&\leq& y_n(F,1-\mu,\mu)\leq \cdots\leq y_0(F,1-\mu,\mu)=F(\mu)
 \end{eqnarray*}

 \end{enumerate}
\end{theorem}

Applying the Theorem \ref{Far}, we obtain the following refinement.

\begin{theorem}
 Let $A,B, X$ be operators such that $A, B\in \+$ and $\alpha, \beta\in [0,1]$
and $|||\cdot|||$ be a unitarily invariant norm. Then
\begin{eqnarray*}
 F\left(\frac{\alpha+\beta}{2}\right)\leq l(\lambda)\leq \frac{1}{b-a}\int_a^b
F(x)dx\leq L(\lambda) \leq \frac{F(\alpha)+F(\beta)}{2}
\end{eqnarray*}
for all $\lambda \in [0,1]$, where
$$
l(\lambda)=\lambda F\left(\frac{\lambda \beta
+(2-\lambda)\alpha}{2}\right)+(1-\lambda) F\left(\frac{(1+\lambda) \beta
+(1-\lambda)\alpha}{2}\right)
$$
 and
$$
L(\lambda)=\frac12 (F(\lambda \beta +(1-\lambda)\alpha)+\lambda
F(\alpha)+(1-\lambda)F(\beta)).
$$
\end{theorem}

Finally, using the refinement presented in Theorem \ref{W} we get the following statement.

\begin{theorem}
Let $A,B, X$ be operators such that $A, B\in \+$. For $a,b \in (0,1)$ with $a<b$ let $T_t$ be the mapping defined in $[0,1]$ by
$$
T_t(a,b)=\frac 12 \left(F\left(\frac{1+t}{2} a+
\frac{1-t}{2}b\right)+F\left(\frac{1-t}{2} a
+\frac{1+t}{2}b\right)\right).
$$

Then, there exists $\xi\in (0,1)$ such that for any $\mu \in (0,1)$ and any unitary invariant norm $|||\cdot|||$,
\begin{eqnarray*}
 2|||A^{1/2}XB^{1/2}|||&\leq& T_{\eta}(\mu, 1-\mu) \leq
T_{\xi}(\mu, 1-\mu)=\frac{1}{1-2\mu}\int_{\mu}^{1-\mu} F(x)dx \nonumber\\
&\leq& T_{\lambda}(\mu, 1-\mu)\leq |||A^{\mu}XB^{1-\mu}+A^{1-\mu}XB^{\mu}|||\,,
\end{eqnarray*}
where $\eta\in [0,\xi]$ and $\lambda\in[\xi,1].$

\end{theorem}

From the generalization of the H-H inequality due to Vasi\'c and Lackovi\'c, we
get
\begin{theorem}
Let $A,B, X$ be operators such that $A, B\in \+$ and let $p, q$ be
positive numbers and $0\leq \alpha< \beta \leq 1.$ Then the double
inequality
\begin{eqnarray*}
F\left(\frac{p\alpha+q\beta}{p+q}\right)\leq
\frac{1}{2y}\int_{c-y}^{c+y}F(t)dt\leq
\frac{pF(\alpha)+qF(\beta)}{p+q}
\end{eqnarray*}
holds for $c=\frac{p\alpha+q\beta}{p+q}$, $y>0$ if and only if
$y\leq \frac{\beta-\alpha}{p+q}\min\{p,q\}.$
\end{theorem}

\section{Refinement of the Heinz inequality for matrices}

In what follows, the capital letters $A, B, X, \cdots$ denote
arbitrary elements of $\mathcal{M}_n$. By $\mathbb{P}_{n}$ we denote
the set of positive definite matrices. The Schur product of two
matrices $A=[a_{ij}]$ and $B=[b_{ij}]$ in $M_{n}$ is the entrywise
product and denoted by $A\circ B$. We shall state the following
preliminary result, which is needed to prove our main results.

If $X=[x_{ij}]$ is positive semidefinite, then for any matrix $Y,$
we have
\begin{eqnarray}\label{had}
|||X\circ Y|||\leq |||Y|||\max_{i} x_{ii}
\end{eqnarray}
for every unitarily invariant norm $|||\cdot|||$. For a proof of this, the reader may be referred to \cite{4}.
\begin{theorem}\label{t1}
Let $A,B\in \mathbb{P}_{n}$ and $X \in M_{n}$. Then for any real numbers $\alpha, \beta$ and any unitarily invariant norm $|||\cdot|||$,
\begin{eqnarray}\label{main1}
&&\hspace{-1.5cm}\left|\left|\left|A^{\frac{\alpha+\beta}{2}}XB^{1-\frac{\alpha+\beta}{2}}+A^{1-\frac{\alpha+\beta}{2}}XB^{\frac{\alpha+\beta}{2}}\right|\right|\right|\leq\frac{1}{|\beta-\alpha|}
\left|\left|\left|\int_{\alpha}^{\beta}\left(A^{\nu}XB^{1-\nu}+A^{1-\nu}XB^{\nu}\right)d\nu\right|\right|\right|\nonumber\\
&\leq& \frac{1}{2}\left|\left|\left|A^{\alpha}XB^{1-\alpha}+A^{1-\alpha}XB^{\alpha}+A^{\beta}XB^{1-\beta}+A^{1-\beta}XB^{\beta}\right|\right|\right|.
\end{eqnarray}
\end{theorem}
\begin{proof}
Without loss of generality assume that $\alpha <\beta$. We shall first prove the result for the case $A=B$. Since the norms considered here are unitarily invariant, so we can assume that $A$ is diagonal, i.e. $A= {\rm diag}(\lambda_{1},\lambda_{2},\cdots ,\lambda_{n}).$\\
Note that
$$A^{\frac{\alpha+\beta}{2}}XA^{1-\frac{\alpha+\beta}{2}}+A^{1-\frac{\alpha+\beta}{2}}XA^{\frac{\alpha+\beta}{2}}=Y\circ \left(\int_{\alpha}^{\beta}\left(A^{\nu}XA^{1-\nu}+A^{1-\nu}XA^{\nu}\right)d\nu\right),$$
where $Y$ is a Hermitian matrix. If $X=[x_{ij}]$ and $Y=[y_{ij}]$, then
$$\left[\lambda_i^{\frac{\alpha+\beta}{2}}x_{ij}\lambda_j^{1-\frac{\alpha+\beta}{2}}+\lambda_i^{1-\frac{\alpha+\beta}{2}}x_{ij}\lambda_j^{\frac{\alpha+\beta}{2}}\right]=\left[y_{ij}\int_{\alpha}^{\beta}\left(\lambda_i^{\nu}x_{ij}\lambda_j^{1-\nu}+\lambda_i^{1-\nu}x_{ij}\lambda_j^{\nu}\right)d\nu\right]\,,$$
whence
\begin{eqnarray*}
y_{ij}&=&\frac{\lambda_i^{\frac{\alpha+\beta}{2}}\lambda_j^{1-\frac{\alpha+\beta}{2}}+\lambda_i^{1-\frac{\alpha+\beta}{2}}\lambda_j^{\frac{\alpha+\beta}{2}}}{\int_{\alpha}^{\beta}\left(\exp\left(\log(\lambda_i)\nu
+\log(\lambda_j)(1-\nu)\right)+\exp\left(\log(\lambda_i)(1-\nu)+\log(\lambda_j)\nu\right)\right)d\nu}\\
&=&\frac{\lambda_i^{\frac{\beta-\alpha}{2}}\left(\lambda_i^\alpha\lambda_j^{1-\beta}+\lambda_i^{1-\beta}\lambda_j^\alpha\right)\lambda_j^{\frac{\beta-\alpha}{2}}(\log\lambda_i-\log\lambda_j)}
{\lambda_i^{\beta}\lambda_j^{1-\beta}-\lambda_i^{1-\beta}\lambda_j^{\beta}-\lambda_i^{\alpha}\lambda_j^{1-\alpha}+\lambda_i^{1-\alpha}\lambda_j^{\alpha}}\\
&=&\frac{\lambda_i^{\frac{\beta-\alpha}{2}}(\log\lambda_i-\log\lambda_j)\lambda_j^{\frac{\beta-\alpha}{2}}}
{\lambda_i^{\beta-\alpha}-\lambda_j^{\beta-\alpha}}\,,\qquad\qquad
{\rm for~} i\neq j
\end{eqnarray*}
and $y_{ii}=\frac{1}{\beta-\alpha}>0$. By \eqref{had}, it is enough to show that the matrix $Y$ is positive semidefinite, or equivalently the matrix
$$y_{ij}' =
\begin{cases}
\frac{\log\lambda_i-\log\lambda_j}
{\lambda_i^{\beta-\alpha}-\lambda_j^{\beta-\alpha}} & \text{if } i\neq j \\
\frac{1}{(\beta-\alpha)\lambda_i^{\beta-\alpha}} & \text{if } i=j
\end{cases}
$$
is positive semidefinite. On taking $\lambda_{i}^{\beta-\alpha}=s_{i},$ we get
$$(\beta-\alpha)y_{ij}' =
\begin{cases}
\frac{\log s_i-\log s_j}
{s_i-s_j} & \text{if } i\neq j \\
\frac{1}{s_i} & \text{if } i=j\,,
\end{cases}
$$
which is a positive semidefinite matrix, since the matrix on the right hand side is the L\"owner matrix corresponding to the matrix monotone function $\log x$; see \cite[Theorem 5.3.3]{2}. This proves the first inequality in \eqref{main1} for the case
$A=B$.

The second inequality will follow on the same lines. We indeed have
$$\int_{\alpha}^{\beta}\left(A^{\nu}XA^{1-\nu}+A^{1-\nu}XA^{\nu}\right)d\nu = Z \circ \left(A^{\alpha}XB^{1-\alpha}+A^{1-\alpha}XB^{\alpha}+A^{\beta}XB^{1-\beta}+A^{1-\beta}XB^{\beta}\right)\,,$$
where $Z$ is the Hermitian matrix with entries
$$z_{ij}=
\begin{cases}
\frac{\lambda_i^{\beta-\alpha}-\lambda_j^{\beta-\alpha}}
{(\log\lambda_i-\log\lambda_j)(\lambda_i^{\beta-\alpha}+\lambda_j^{\beta-\alpha})} & \text{if } i\neq j \\
\frac{(\beta-\alpha)}{2} & \text{if } i=j\,.
\end{cases}
$$
On taking $\lambda_{i}^{\beta-\alpha}=e^{t_i}$ we conclude that $Z$ is positive semidefinite if and only if so is the following matrix
$$\frac{2}{\beta-\alpha}z_{ij}' =
\begin{cases}
\frac{\tanh((t_i-t_j)/2)}{(t_i-t_j)/2} & \text{if } i\neq j \\
1& \text{if } i=j\,.
\end{cases}
$$
The right hand side matrix is positive semidefinite since the function $f(x)=\frac{\tanh x}{x}$ is positive definite; see \cite[Example 5.2.11]{2}.
This proves the second inequality in \eqref{main1} for the case $A=B$.\\
The general case follows on replacing $A$ by
$\left[\begin{array}{cc}
     A & 0 \\
     0 & B \\
\end{array}\right]$ and $X$ by $ \left[\begin{array}{cc}
     0 & X \\
     0 & 0 \\
\end{array}\right].$
\end{proof}

The first corollary provides some variants of \cite[Theorem 2 and Theorem 3]{9}. It should be noticed that
$$\lim_{\mu\to 1/2}\left(\frac{2}{|1-2\mu|}\left|\left|\left|\int_{\mu}^{1/2}(A^{\nu}XB^{1-\nu}+A^{1-\nu}XB^{\nu})d\nu\right|\right|\right|\right)=2\left|\left|\left|A^{1/2}XB^{1/2}\right|\right|\right|$$
and
$$\lim_{\mu\to 0}\left(\frac{1}{|\mu|}\left|\left|\left|\int_{0}^{\mu}(A^{\nu}XB^{1-\nu}+A^{1-\nu}XB^{\nu})d\nu\right|\right|\right|\right)=|||AX+XB|||\,.$$
\begin{corollary}
Let $A,B\in \mathbb{P}_{n}$, $X \in M_{n}$, $\mu$ be a real number and $|||\cdot|||$ be any unitarily invariant norm. Then
\begin{eqnarray*}
&&\hspace{-1cm}\left|\left|\left|A^{\frac{2\mu+1}{4}}XB^{\frac{3-2\mu}{4}}+A^{\frac{3-2\mu}{4}}XB^{\frac{2\mu+1}{4}}\right|\right|\right|\leq \frac{2}{|1-2\mu|}\left|\left|\left|\int_{\mu}^{1/2}(A^{\nu}XB^{1-\nu}+A^{1-\nu}XB^{\nu})d\nu\right|\right|\right|\\
&& \hspace{1in}\leq \frac{1}{2}\left|\left|\left|A^{\mu}XB^{1-\mu}+A^{1-\mu}XB^{\mu}+2A^{1/2}XB^{1/2}\right|\right|\right|\,,
\end{eqnarray*}
\begin{eqnarray*}
&&\hspace{-1cm}\left|\left|\left|A^{\frac{\mu}{2}}XB^{1-\frac{\mu}{2}}+A^{1-\frac{\mu}{2}}XB^{\frac{\mu}{2}}\right|\right|\right|\leq
\frac{1}{|\mu|}\left|\left|\left|\int_{0}^{\mu}(A^{\nu}XB^{1-\nu}+A^{1-\nu}XB^{\nu})d\nu\right|\right|\right|\\
&&\hspace{1in} \leq \frac{1}{2}\left|\left|\left|AX+XB+A^{\mu}XB^{1-\mu}+A^{1-\mu}XB^{\mu}\right|\right|\right|\,.
\end{eqnarray*}
\end{corollary}

The following consequence provides a matrix analogue of \eqref{Heinz}.
\begin{corollary}
Let $A,B\in \mathbb{P}_{n}$ and $X\in M_{n}.$
Then for any $0\leq \alpha<\beta \leq 1$ with $\alpha+\beta\leq 2$ and any unitarily invariant norm $|||\cdot|||$,
\begin{eqnarray*}
2|||A^{1/2}XB^{1/2}|||&\leq& \left|\left|\left|A^{\frac{\alpha+\beta}{2}}XB^{1-\frac{\alpha+\beta}{2}}+A^{1-\frac{\alpha+\beta}{2}}XB^{\frac{\alpha+\beta}{2}}\right|\right|\right|\\
&\leq& \frac{1}{|\beta-\alpha|}
\left|\left|\left|\int_{\alpha}^{\beta}\left(A^{\nu}XB^{1-\nu}+A^{1-\nu}XB^{\nu}\right)d\nu\right|\right|\right|\\
&\leq&  \frac{1}{2}\left|\left|\left|A^{\alpha}XB^{1-\alpha}+A^{1-\alpha}XB^{\alpha}+A^{\beta}XB^{1-\beta}+A^{1-\beta}XB^{\beta}\right|\right|\right|\\
&\leq&  \frac{1}{2}\left|\left|\left|A^{\alpha}XB^{1-\alpha}+A^{1-\alpha}XB^{\alpha}\right|\right|\right|+\frac{1}{2}\left|\left|\left|A^{\beta}XB^{1-\beta}+A^{1-\beta}XB^{\beta}\right|\right|\right|\\
&\leq& |||AX+XB|||.
\end{eqnarray*}
\end{corollary}
\begin{proof} Applying the triangle inequality, the properties of the function $f(\nu)=|||A^{\nu}XB^{1-\nu}+A^{1-\nu}XB^{\nu}|||$ and
Theorem \ref{t1} we get the required inequalities.
\end{proof}

It is shown in \cite[Corollary 3]{9} that
\begin{eqnarray}\label{r0}
|||A^{\nu}XB^{1-\nu}+A^{1-\nu}XB^{\nu}|||\leq 4r_{0}||| A^{1/2}XB^{1/2}|||+(1-2r_{0})|||AX+XB|||.
\end{eqnarray}
A natural generalization of \eqref{r0} would be
$$|||A^{\nu}XB^{1-\nu}+A^{1-\nu}XB^{\nu}|||\leq ||| 4r_{0} A^{1/2}XB^{1/2}+(1-2r_{0})(AX+XB)|||$$
for $0\leq\nu \leq 1$ and $r_{0}=\min\{\nu, 1-\nu\}$ with $A, B \in
\mathbb{P}_{n}$ and $X\in M_{n},$ which in fact is not true, in
general. The following counterexample justifies this:

Take $X=\left[
\begin{array}{ccc}
52.39 & 38.71 & 12.36 \\
32.86 & 35.38 & 64.82 \\
91.79 & 99.45 & 66.10 \\
\end{array}
\right],$ $A=\left[\begin{array}{ccc}
92.315 & 87.791 & 71.090 \\
87.791 & 120.130 & 83.340 \\
71.090& 83.340 & 103.610 \\
\end{array}
\right]$,\\
$B=\left[\begin{array}{ccc}
118.482 & 23.249 & 112.676 \\
23.249 & 10.343 & 38.224\\
112.676 & 38.224 & 156.551 \\
\end{array}
\right]$ and $\nu=0.4680$. Then  ${\rm
tr}|A^{\nu}XB^{1-\nu}+A^{1-\nu}XB^{\nu}|=78135.5,$ while ${\rm
tr}|4r_{0} A^{1/2}XB^{1/2}+(1-2r_{0})(AX+XB)|= 78125.4$.

We shall, however, present another result, which is a possible generalization of \eqref{r0}.\\

\begin{theorem}{\label{Kit}}
Let $A,B\in \mathbb{P}_{n}$ and $X\in M_{n}.$ Then for $ \nu \in [0,1]$ and for every unitarily invariant norm $|||\cdot|||$,
\begin{eqnarray}\label{gr0}
|||A^{\nu}XB^{1-\nu}+A^{1-\nu}XB^{\nu}|||\leq ||| 4r_{1}(\nu)
A^{1/2}XB^{1/2}+(1-2r_{1}(\nu))(AX+XB)|||\,,
\end{eqnarray}
where $r_{1}(\nu)=\min\{\nu, |\frac 12 -\nu|, 1-\nu\}.$
\end{theorem}
\begin{proof}
First, we consider the case $\nu \in [0,1/2].$ Notice that by some
simple algebraic or geometrical arguments, we may conclude that
$0\leq r_{1} \leq 1/4$. Again, by following a similar way as in
Theorem \ref{t1}, we can write the matrix
$$A^{\nu}XA^{1-\nu}+A^{1-\nu}XA^{\nu}=W \circ (4r_{1} A^{1/2}XA^{1/2}+(1-2r_{1})(AX+XA)),$$
where $W$ is a Hermitian matrix with entries
$$w_{ij}=
\begin{cases}
\frac{\lambda_{i}^{\nu}(\lambda_{i}^{1-2\nu}+\lambda_{j}^{1-2\nu})\lambda_{j}^{\nu}}
{4r_{1}\lambda_{i}^{1/2}\lambda_{j}^{1/2}+(1-2r_{1})(\lambda_{i}+\lambda_{j})} & \text{if } i\neq j \\
1 & \text{if } i=j
\end{cases}
$$
Now, observe that $0\leq \frac{4r_{1}}{1-2r_{1}}\leq 2$ and $0\leq
1-2\nu\leq 1,$ so the matrix $W$ is positive semidefinite; see
\cite[Theorem 5.2, p.225]{3}. On repeating the same argument as in
Theorem \ref{t1}, the required inequality \eqref{gr0} follows.

Finally, if $\nu \in [\frac12,1]$  let $\mu=1-\nu\in [0,\frac 12]$,
then by the previous case we have
\begin{eqnarray}
|||A^\nu XB^{1-\nu}+A^{1-\nu}
XB^{\nu}|||&=&|||A^{1-\mu}XB^{\mu}+A^\mu XB^{1-\mu}|||\nonumber\\
&\leq& |||4r_1(\mu) A^{\frac 12}XB^{\frac
12}+(1-2r_1(\mu))(AX+XB)|||\,, \nonumber
\end{eqnarray}
where $r_1(\mu)=\min\{\mu, |\frac 12-\mu|, 1-\mu\}=r_1(\nu).$
\end{proof}

From the previous theorem, we deduce a new refinement of the Heinz
inequality for matrices.

\begin{corollary}
Let $A,B\in \mathbb{P}_{n}$ and $X\in M_{n}.$ Then for $\nu \in [0,1]$ and for every unitarily invariant norm $|||\cdot|||$,
\begin{eqnarray*}
|||A^{\nu}XB^{1-\nu}+A^{1-\nu}XB^{\nu}|||&\leq& ||| 4r_{1}(\nu) A^{1/2}XB^{1/2}+(1-2r_{1}(\nu))(AX+XB)|||\\
&\leq& 4r_{1}(\nu)|||A^{1/2}XB^{1/2}|||+(1-2r_{1}(\nu))|||AX+XB|||\\
&\leq& 2(2r_{1}(\nu)-1)|||A^{1/2}XB^{1/2}|||+2(1-r_{1}(\nu))|||AX+XB|||\\
&\leq&|||AX+XB|||\,,
\end{eqnarray*}
where $r_{1}(\nu)=\min\{\nu, |\frac 12-\nu|,  1-\nu\}.$
\end{corollary}

As a direct consequence of Theorem \ref{Kit}, we obtain the
following refinement of an inequality (see \cite{cms}).

\begin{corollary}
Let $A,B\in \mathbb{P}_{n},$ $X\in M_{n},$ $r\in[\frac 12, \frac
32]$ and $t\in(-2,2]$. Then  for every unitarily invariant norm
$|||\cdot|||$,
\begin{eqnarray*}
|||A^{r}XB^{2-r}+A^{2-r}XB^{r}|||&\leq& ||| 4s
AXB+(1-2s)(A^{3/2}XB^{1/2}+A^{1/2}XB^{3/2})|||\\
&\leq&  4s
|||AXB|||+(1-2s)|||A^{3/2}XB^{1/2}+A^{1/2}XB^{3/2}|||\\
&\leq& 4s
|||AXB|||+(1-2s)\frac{2}{t+2}|||A^2X+tAXB+XB^2|||\\
&\leq& 2(2s-1)|||AXB|||+\frac{4(1-s)}{t+2}|||A^2X+tAXB+XB^2|||\\
&\leq&\frac{2}{t+2}|||A^2X+tAXB+XB^2|||
\end{eqnarray*}
in which $s=\min\{r-\frac 12, |1-r|,  \frac32-r\}.$
\end{corollary}
\begin{proof}
Let $Y=A^{1/2}XB^{1/2}\in M_n$ and $\nu=r-\frac 12\in [0,1]$. It
follows from Theorem \ref{Kit} that
\begin{eqnarray*}
|||A^{r}XB^{2-r}+A^{2-r}XB^{r}|||&=&|||A^{r}A^{-1/2}YB^{-1/2}B^{2-r}+A^{2-r}A^{
-1/2}YB^{-1/2}B^{r} |||\\
&=&|||A^{\nu}YB^{1-\nu}+A^{1-\nu}YB^{1-\nu}|||\\
&\leq&||| 4r_{1}(\nu) A^{1/2}YB^{1/2}+(1-2r_{1}(\nu))(AY+YB)|||\\
&=& ||| 4r_{1}(\nu)
AXB+(1-2r_{1}(\nu))(A^{3/2}XB^{1/2}+A^{1/2}XB^{3/2})|||\,,
\end{eqnarray*}
where $r_1(\nu)=\min\{\nu, |\frac 12-\nu|,  1-\nu\}.$ Let
$s=r_1(r-\frac12)$. Applying the triangle inequality and Zhan's
inequality, we obtain
\begin{eqnarray*}
|||A^{r}XB^{2-r}+A^{2-r}XB^{r}|||&\leq& ||| 4s AXB+(1-2s)(A^{3/2}XB^{1/2}+A^{1/2}XB^{3/2})|||\\
&\leq& 4s|||AXB|||+(1-2s)|||A^{3/2}XB^{1/2}+A^{1/2}XB^{3/2}|||\\
&\leq& 4s|||AXB|||+\frac{2(1-2s)}{t+2}|||A^2X+tAXB+XB^2|||\\
&\leq& 2(2s-1)|||AXB|||+\frac{4(1-s)}{t+2}|||A^2X+tAXB+XB^2|||\\
&\leq&\frac{2}{t+2}|||A^2X+tAXB+XB^2|||.\\
\end{eqnarray*}
\end{proof}

\end{document}